\documentclass[11pt, oneside]{article}
\usepackage{geometry}                		
\usepackage[all,cmtip]{xy}
\usepackage{amsmath, amsthm, amssymb}
\usepackage{eucal, mathrsfs, yfonts}
\usepackage{amsfonts}
\usepackage{array}
\usepackage{hyperref}
\usepackage{graphicx}
\usepackage{caption}
\usepackage{tikz}
\usetikzlibrary{calc}
\usepackage{verbatim}
\newtheorem{thm}{Theorem}[section]

\newtheorem{prop}[thm]{Proposition}
\newtheorem{cor}[thm]{Corollary}
\theoremstyle{definition}

\theoremstyle{remark}
\newtheorem{rem}[thm]{Remark}

\numberwithin{equation}{section}
\newcommand{\E}{{\mathcal E}}

\providecommand{\abs}[1]{\lvert#1\rvert}

\def\H{\mbb H}
\def\E{\mathcal E}

\def\H{\mathbb H}
\def\X{X}
\def\T{{T}}

\def\ZZ{{\mathbb Z}}

\def\Ordo{\mathcal O}

\usepackage{graphicx}				
\usepackage{amssymb}

\begin{document}
\title{The Kusuoka measure and the energy Laplacian on \\ level-$k$ Sierpi\'nski gaskets}

\author{Anders \"Oberg and Konstantinos Tsougkas}

\date{}
\maketitle

\begin{abstract}\noindent
We extend and survey results in the theory of analysis on fractal sets from the standard Laplacian on the Sierpi\'nski gasket to the energy Laplacian, which is defined weakly by using the Kusuoka energy measure. We also extend results from the Sierpi\'nski gasket to level-$k$ Sierpi\'nski gaskets, for all $k\geq 2$. We observe that the pointwise formula for the energy Laplacian is valid for all level-$k$ Sierpi\'nski gaskets, $SG_k$, and we provide a proof of a known formula for the renormalization constants of the Dirichlet form for post-critically finite self-similar sets along with a probabilistic interpretation of the Laplacian pointwise formula. We also provide a vector self-similar formula and a variable weight self-similar formula for the Kusuoka measure on $SG_k$, as well as a formula for the scaling of the energy Laplacian.
\end{abstract}

\section{Introduction}\noindent
A theory of analysis on self-similar fractals is currently being developed. The main focus is the Laplace operator, as defined by Kigami, see \cite{kigami}, \cite{kigami2} and \cite{kigami3}. As a prototype, many authors focus on a specific fractal called the Sierpi\'nski gasket. Detailed expositions can be found in \cite{kigami}, \cite{str} and \cite{teplyaevenergy}. In the theory of analysis on fractal sets, a Laplacian is usually defined
weakly with respect to an invariant measure on the fractal set $K$. A
standard way of accomplishing this is to first define a Dirichlet
energy form $\E(f,f)$ on the fractal $K$, in analogy with
$\int |\nabla f|^2 \; d\mu$, and then to define the Laplacian by equating
the corresponding bilinear form $\E (u,v)$ with $-\int_K
\left(\Delta_{\mu}u\right)v \; d\mu$, for functions $v$ vanishing on the
boundary of $K$. The choice of the measure $\mu$ becomes a
delicate question.  It is well-known that with respect to the uniform measure on $K$, the domain of the
Laplacian is not even closed under multiplication \cite{domain}. By
contrast, the Kusuoka measure is well-behaved in this sense, at least for the $L^2$ domain of the Laplacian, and in
some subtler ways, for example the Laplacian it defines has
Gaussian heat kernel estimates with respect to the effective
resistance metric, and the energy Laplacian can furthermore be regarded as a second order differential
operator \cite{kigami4}. 

\newcommand*\rows{6}
\begin{figure}
\centering
\begin{tikzpicture}[scale = 1.3]
    \foreach \row in {0, 1, ...,2} {
        \draw ($\row*(0.5, {0.5*sqrt(3)})$) -- ($(2,0)+\row*(-0.5, {0.5*sqrt(3)})$);
        \draw ($\row*(1, 0)$) -- ($(2/2,{2/2*sqrt(3)})+\row*(0.5,{-0.5*sqrt(3)})$);
        \draw ($\row*(1, 0)$) -- ($(0,0)+\row*(0.5,{0.5*sqrt(3)})$);
    }
\end{tikzpicture}
\begin{tikzpicture}[scale = 1]
    \foreach \row in {0, 1, ...,3} {
        \draw ($\row*(0.5, {0.5*sqrt(3)})$) -- ($(3,0)+\row*(-0.5, {0.5*sqrt(3)})$);
        \draw ($\row*(1, 0)$) -- ($(3/2,{3/2*sqrt(3)})+\row*(0.5,{-0.5*sqrt(3)})$);
        \draw ($\row*(1, 0)$) -- ($(0,0)+\row*(0.5,{0.5*sqrt(3)})$);
    }
\end{tikzpicture}
\begin{tikzpicture}[scale = 0.5]
    \foreach \row in {0, 1, ...,\rows} {
        \draw ($\row*(0.5, {0.5*sqrt(3)})$) -- ($(\rows,0)+\row*(-0.5, {0.5*sqrt(3)})$);
        \draw ($\row*(1, 0)$) -- ($(\rows/2,{\rows/2*sqrt(3)})+\row*(0.5,{-0.5*sqrt(3)})$);
        \draw ($\row*(1, 0)$) -- ($(0,0)+\row*(0.5,{0.5*sqrt(3)})$);
    }
\end{tikzpicture}
\caption{The $\Gamma_1$ network of $SG_2$, $SG_3$ and $SG_6$}
\end{figure}
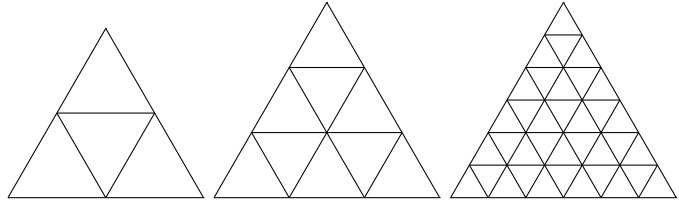
In this paper we study properties of the Kusuoka measure and some properties for the Laplacian it defines
on a family of fractals, the level-$k$ Sierpi\'nski gaskets, $SG_k$, which are realized in
${\mathbb R}^2$ and constructed by $k(k+1)/2$ contraction mappings $F_i(x)=x/k  +b_{i,k}$
for appropriate choices of $b_{i,k}$, so that $SG_k$ is the unique nonempty compact set that satisfies the self-similar identity
$$SG_k=\bigcup_{i=0}^{\frac{1}{2}(k+2)(k-1)} F_i(SG_k).$$
Notice that $SG_2$ is the ordinary Sierpi\'nski gasket $SG$. These fractals satisfy the open set condition and therefore their Hausdorff dimension can easily be calculated by Moran's equation which gives us that each $SG_k$ has Hausdorff dimension
$$s=1+\frac{\log{(k+1)}-\log{2}}{\log{k}}.$$
In the limit their dimension tends to two, the dimension of $\mathbb{R}^2$.
For any $SG_k$, if $w=(w_1,\ldots,w_n)$ is a finite word, we define the mapping
$$F_w=F_{w_1}\circ \cdots \circ F_{w_n}.$$
We call $F_w K$ a cell of level $m$ of $SG_k$. Any $SG_k$ may be approximated by a sequence of graphs $\Gamma_n$ with a vertex set $V_n$ and edge relations $x\sim_n y$. That means, that for $x,y \in V_n$ we have that 
$$ x\sim_n y \iff x,y \in F_w(V_0)$$ for some word $w$ of length $n$. The set $V_0=\{q_0,q_1,q_2\}$, which is the same for all $SG_k$, will be regarded as the the boundary of $SG_k$. We also define $V^*=\bigcup_{n=0}^{\infty} V_n$. If $x \in V^*\diagdown V_0$, then $x$ is called a junction point. We can derive an explicit formula for the number of vertices in $V_n$ for any $k$. We know that the number of vertices for self-similar graphs is $|V_n|=\frac{m^n(|V_1|-|V_0|)+m|V_0|-|V_1|}{m-1}$ where $m$ is the number of contractions. In $SG_k$ we have that $m=\frac{k(k+1)}{2}$ and it is easily seen that $V_1^k=\frac{(k+1)(k+2)}{2}$.
Then substituting to the above equation we have that
$$|V_n^k|=\frac{\left( \frac{k(k+1)}{2}\right)^n (k+4)+2(k+1)}{(k+2)}.$$

The standard invariant measure $\mu$ on $SG_k$, which is a normalized version of the s-dimensional Hausdorff measure, satisfies
 $$\mu(F_wF_i SG_k)=\frac{2}{k(k+1)}\mu(F_w SG_k), \quad i=0,1,\ldots , k-1,
\text{   for any word } w.$$
For such a measure $\mu$, we also have a self-similar identity 
$$\mu(A)= \sum_i \mu_i \mu(F_i^{-1}A),$$for any Borel subset $A$ of $SG_k$ whereas the Kusuoka measure satisfies a more 
complicated invariance identity and is not a self-similar measure. In \cite{str3} it was proved that for $SG_2$ we have for the Kusuoka measure $\nu$ and any Borel subset $A$ the following scaling relationship
\begin{equation}\label{eq: selfsim}
\nu(F_i A)=\int_A p_i(x)\; d\nu(x),
\end{equation}
$i=0, 1, 2$, and where
\begin{equation}\label{eq: weights}
p_i(x)=\frac{1}{15}+\frac{12}{15}\cdot \frac{d\nu_i}{d\nu},
\end{equation}
where $\nu=\nu_0+\nu_1+\nu_2$, where each $\nu_i$ is an energy measure defined by the energy form ${\mathcal E}$ with respect to a standard basis of harmonic functions $h_i$, see Section $2$ for details. In \cite{tsougkas} it was proved that the harmonic structure is non-degenerate for all $SG_k$ and this makes it possible to study energy measures on these sets as well. More precisely, for any Borel set $A$, we can define an energy measure for a particular harmonic function $h_i$ by
\begin{equation}\label{energy_h}
\nu_i (A)=\nu_{h_i}(A)= \lim_{m\to \infty}\frac{1}{r^m}\sum_{\{x,y\in A; \; x\sim_m y\}} \left(h_i(x)-h_i(y)\right)^2,
\end{equation}
where $r$ is the renormalization factor, or the resistance in the electric network interpretation, for $SG_k$, and which is different for different $k$. For $k=2$ we have $r=3/5$ and for $k=3$ we have $r=7/15$. 
We note that the energy form ${\mathcal E}$ is the limit of energy forms ${\mathcal E}_m$ that are defined for functions on the vertex set $V_m$ and which extend to the next $V_{m+1}$ by means of minimizing energy. In the case of harmonic functions, the energy forms 
$${\mathcal E}_m(h,h)=r^{-m} \sum_{x\sim_m y} (h(x)-h(y))^2$$ are constant in $m$. For a general function $u$ on $V^*$, we have ${\mathcal E}_{m+1}(u,u)\geq  {\mathcal E}_m (u,u)$, so the limit ${\mathcal E}(u,u)=\lim_{m\to \infty} {\mathcal E}_m(u,u)$ exists and let dom${\mathcal E}=\{u\in V^*: {\mathcal E}(u,u)<\infty\}$. Since the set of approximating vertices $V^*$ is dense in $SG_k$ and since the dom${\mathcal E}$ consists of uniformly continuous functions on $V^*$ \cite{kigami}, we also have dom${\mathcal E}=\{u\in C(SG_k): {\mathcal E}(u,u)<\infty\}$. We let dom$_0{\mathcal E}$ denote functions of finite energy that vanish on the boundary $V_0$. The precise definition of the energy Laplacian $\Delta_\nu$ now becomes that if $u\in$dom${\mathcal E}$, then we say that $u\in$dom$\Delta_\nu$ and $\Delta_\nu u=f$ if ${\mathcal E}(u,v)=-\int_{SG_k}f v\; d\nu$, for all $v\in$dom$_0{\mathcal E}$. 

In section 2, we give a probabilistic interpretation of the pointwise formula for the standard Laplacian as well as provide a different proof of the previously known (see \cite{thesis}) pointwise formula for the energy Laplacian and the renormalization factor (Propositions 2.1, 2.2 and 2.6). We extend to all $SG_k$ the result that the only functions in both the domains of the standard and energy Laplacian are the harmonic functions. As said earlier, we can write the Kusuoka measure $\nu$ on $SG_2$ on the self-similar form \eqref{eq: selfsim} with the weights $p_i$ defined (see \eqref{eq: weights}) in terms of $\frac{d\nu_i}{d\nu}$, the Radon--Nikodym derivative between an energy measure $\nu_i$ and the Kusuoka measure $\nu$. The self-similar formula is in reality somewhat complicated, since these Radon--Nikodym derivatives have dense sets of discontinuity points (see \cite{str3}). However, in \cite{JOP} it is proved that these functions belong to a Banach space that contains both H\"older continuous functions and the weight functions of the form $p_i$, if we study the dynamics on a symbolic space. Instead of studying self-similarity of the Kusuoka measure with respect to ``variable weight functions'' $p_i$, it is possible to write the self-similarity on vector form. In Theorem 2.8 we provide such a formula for all $SG_k$:
\begin{equation}
\label{eq:matrices}
\left(
\begin{array}{c}
\nu_0\\ \nu_1\\ \nu_2
\end{array}
\right) = \sum_{i=0}^{\frac{(k+2)(k-1)}{2}} 
M_i
\left(
\begin{array}{c}
\nu_0\\ \nu_1 \\ \nu_2
\end{array}
\right)\circ F_i^{-1},
\end{equation}
for certain matrices $M_i$ which we describe explicitly. But we also obtain a variable weight self-similar formula, see Corollary 2.9, for all $SG_k$. In addition, we obtain an analogous variable weight self-similar formula for the energy Laplacian for all $SG_k$ (Theorem 2.9), which generalizes the corresponding result in \cite{str3}.
\newline

\noindent
{\bf Acknowledgements}. We are grateful to Robert S.\ Strichartz for valuable suggestions, and for being able to visit Cornell University. We also especially thank the anonymous referee for numerous helpful comments and suggestions that significantly improved this article. The first author was supported in part within the project {\em Ergodic theory of energy measures on fractals}, financed by the {\bf Royal Society (UK)}, grant {\bf IE121546}.

\section{Invariance results for the energy Laplacian and the Kusuoka measure}\noindent
The Kusuoka measure is defined as 
$$\nu={\nu}_{h_0}+{\nu}_{h_1}+{\nu}_{h_2}$$
where $h_i(q_j)=\delta_{ij}$ for $i,j=0,1,2$ is the standard basis of harmonic functions. This definition is valid for any $SG_k$ keeping in mind that the harmonic functions look differently depending on $k$. An equivalent definition for the Kusuoka measure would be to define it as $${\nu}'={\nu}_{h}+{\nu}_{h^{\perp}}$$ where $\{h,h^{\perp}\}$ is an orthonormal basis of harmonic functions modulo constants. This definition is independent of the orthonormal basis used and in this case we get that  ${\nu}'=\frac{1}{3}{\nu}$, or more generally $\nu '=\frac{1}{3}\mathcal{E}(h,h)\nu$ if we drop the orthonormality condition. The Kusuoka measure is singular with respect to the standard measure (\cite{domain},\cite{kusuoka}). For $SG_2$ the decay rate of the Kusuoka measure is $(\frac{3}{5})^m$. Specifically, for all words $w$ of length $\abs{w}=m$ we have that $\nu(F_{w}K)\leqslant c \left(\frac{3}{5}\right)^m$. We can see this as follows. In \cite{str3} it is shown that for any word $w$ and $i=0,1,2$ we get that, $\nu(F_wF_i^mK)=O\left(\frac{3}{5}\right)^m.$ Also, for all harmonic functions $h$ and for any word $w$ of length $m$ we have according to \cite{str1} that 
$$\max_{|w|=m}\nu_h(F_wK)=\max_{i=0,1,2}\nu_h(F_i^mK)$$
and thus 
\begin{equation*}
\begin{split}
\nu(F_{w}K)&=\nu_{h}(F_{w}K)+\nu_{h^{\perp}}(F_{w}K) \leqslant \sup_{i=0,1,2} {\nu_{h}(F_i^mK)} + \sup_{i=0,1,2} {\nu_{h^{\perp}}(F_i^mK)}\\ 
& \leqslant \sup_{i=0,1,2} {\nu(F_i^m K)} + \sup_{i=0,1,2} {\nu(F_i^m K)} \leqslant c\left(\frac{3}{5}\right)^m.
\end{split}
\end{equation*}

\subsection{Pointwise formulas for the self-similar and energy Laplacians}
The graph Laplacians are defined on $V_m$ for $SG_k$ as 
\begin{equation}
\begin{split}
&\Delta_m u(x)=\frac{1}{deg(x)}\sum_{y\sim_m x} (u(y)-u(x)) \text{ for } x \in V_m\setminus{V_0}.
\end{split}
\end{equation}
where $deg(x)$ is the cardinality of the set $\{y\sim_{m} x\}$. We may also define the normal derivative of a boundary point as
$$\partial_n u(x)=\lim_{m\to \infty} \sum_{y\sim_m x} (u(x)-u(y))$$
which can also be localized to junction points by $\partial_n u(F_wq_i)=r^{-|w|}\partial_n(u\circ F_w)(q_i)$. Now let the functions $\psi_x^{(m)}$, which are called piecewise harmonic splines, to be $\psi_x^{(m)}(y)=\delta_{xy}$ for $y\in V_m$. From \cite{str1}, assuming a regular harmonic structure with all edge weights equal to $1$, the pointwise formula for $x \in V^* \diagdown V_0$ is
\begin{equation}
\label{eq:pointwise}
\Delta_{\mu} u(x)=\lim\limits_{m \rightarrow \infty}{r^{-m}\left( \int_K {\psi_x^{(m)}} \mathrm {d\mu}\right)}^{-1}deg(x)\Delta_m u(x).
\end{equation}
It suffices then to know the renormalization (or scaling) constant for the energy form (the resistance, in the electric network interpretation) and $\int_K {\psi_x^{(m)}} \mathrm {d\mu}$. Although a closed formula is unknown for the renormalization constant in general, even for $SG_k$, we can nevertheless provide a formula with a probabilistic interpretation which is based on random walks on the $\Gamma_m$ graphs. This formula appears without proof in \cite{tep} (p.\ 38) and we have not been able to find a proof of it in the literature.
\begin{prop}
For a fully symmetric p.c.f.\ set, with a regular harmonic structure, the renormalization constant satisfies the formula $$r=1-p$$ where $p$ is the probability that a simple symmetric random walk starting at a boundary vertex, and killed at the remaining boundary vertices, returns to itself.
\end{prop}
\begin{proof}
Full symmetry implies that any permutation of the boundary $V_0=\{q_0,\dots, q_n\}$ extends to an isometry of the fractal. Theorem A.1.2 of \cite{kigami} gives us that $r$ is the largest eigenvalue strictly less than $1$ of the harmonic extension matrix $A_0$ of the cell containing $q_0$. The symmetry of the Laplacian with respect to maps permuting the boundary points implies that the eigenvectors of $A_0$ must be symmetric with respect to those maps and therefore, by normalizing and up to an additive constant, may be taken to be $0$ at $q_0$ and $1$ at $q_1$, and the appropriate values at the other boundary points will depend on the corresponding symmetry. Then the value of the harmonic function with these boundary points at $F_0(q_1)$ is equal to the corresponding eigenvalue. This value is largest when the boundary values $V_0 \setminus \{q_0\}$ have full symmetry, which makes them all equal to $1$ in which case the eigenvalue is $r$. Probabilistically this means that the value for the harmonic function at $F(q_1)$ is the probability that the random walk from $F_0(q_1)$ hits first $V_0 \setminus \{q_0\}$, which is equivalent to the probability that the random walk starting from $q_0$ killed at $V_0 \setminus \{q_0\}$ fails to return to $q_0$ before reaching $V_0 \setminus \{q_0\}$.

\end{proof}
In $SG_k$ it is obvious that as $k\rightarrow \infty$ we have that $r_k \rightarrow 0$. In fact it is easy to see that the renormalization constant satisfies $$r_k>\frac{2}{3k}.$$ Indeed, let $q_1$ and $q_2$ be two boundary vertices of the $V_1$ network of $SG_k$. We put along each edge resistance $1$ and apply $\Delta -Y$ transformations until we arrive to its $V_0$ network where each edge will have resistance $\frac{1}{r_k}$. Then, the effective resistance $R(q_1,q_2)$ remains the same which is easily seen to be $R(q_1,q_2)=\frac{2}{3r_k}$ and thus $r_k=\frac{2}{3R(q_1,q_2)}$. However, it is known that for any graph, if there exists more than one path from $v_1$ to $v_2$ ,which is now the case, then $R(q_1,q_2)<d(q_1,q_2)$ where $d$ is the standard graph distance. Thus $r_k>\frac{2}{3k}$. A more detailed analysis has been performed in \cite{uta,hambly} showing that there exist $c_1, c_2>0$ such that $c_1 \log{k} \leqslant 1/r_k \leqslant c_2\log{k}$. 

We are now interested in obtaining a probabilistic formula for the Laplacian. In fact, the following formula is valid not just for the Sierpi\'nski gaskets, but for other fully symmetric p.c.f.\ sets. If we have a simple symmetric random walk on a graph, we call the expected hitting time $H(q_1,q_2)$ of two vertices, the expected number of steps that the random walk starting from $q_1$ first arrives at $q_2$. The commute time is $K(q_1,q_2)=H(q_1,q_2)+H(q_2,q_1)$. The pointwise formula for the self-similar Laplacian can then be evaluated as follows.

\begin{prop}
For a fully symmetric p.c.f.\ nested self-similar fractal with a regular harmonic structure, the pointwise formula for the Laplacian is given by
$$\Delta_{\mu} u(x)=|V_0|(|V_0|-1) \lim\limits_{m \rightarrow \infty}{\left( \frac{H(q_1,q_2)}{|V_0|-1} \right)}^m \Delta_m u(x)$$
where $H(q_1,q_2)$ is evaluated on the first graph approximation.
\end{prop}
\begin{proof}
Let $x \in V_m$ for some $m$. In the sequence of approximating graphs, we have that $\Gamma_0$ is the complete graph on the vertex set $V_0$ and, since the effective resistance between two vertices $q_1,q_2$ on the complete graph is $\frac{2}{|V_0|}$, we obtain that $r=\frac{2}{|V_0|R(q_1,q_2)}$. It is also known, see for example \cite{lova}, that the effective resistance between two vertices $q_1,q_2$ of a graph satisfies the formula $R(q_1,q_2)=\frac{K(q_1,q_2)}{2|E|}$ where $K$ is the commute time and $|E|$ is the number of edges. If we take $q_1,q_2$ to be on $V_0$, the commute time is symmetric and it is twice the expected hitting time $H(q_1,q_2)$. On $\Gamma_m$ the number of edges is $\frac{|V_0|(|V_0|-1)}{2}N^m$, where $N$ is the number of contractions, which gives us $r=\frac{(|V_0|-1)N}{H(q_1,q_2)}$. As for the evaluation of the integrals of the piecewise harmonic splines on an $m$-cell containing $x$ we  can now use the following observation. We write $x=x_1$ and take $\psi_{x_1}^{(m)},\psi_{x_2}^{(m)},\dots, \psi_{x_{|V_0|}}^{(m)}$, each for a boundary vertex of the $m$-cell, and observe that 
$$\int_{F_wK}\sum_{i=1}^{|V_0|} {\psi_{x_i}^{(m)}} \mathrm {d\mu}=\int_{F_wK} 1 \mathrm {d\mu}=\frac{1}{N^m}.$$
By symmetry, each of the $\psi_{x_i}^{(m)}$ summands gives the same contribution to the integral which gives us that
$$\int_{F_wK} {\psi_{x}^{(m)}} \mathrm {d\mu}=\frac{1}{|V_0|N^m} \hspace{0.3cm} \text{ and thus } \hspace{0.3cm} \int_K {\psi_x^{(m)}} \mathrm {d\mu}=\frac{deg(x)}{|V_0|(|V_0|-1)N^m}.$$
We obtain the result by substituting the above calculations for $r$ and $\int_K {\psi_x^{(m)}} \mathrm {d\mu}$ into \eqref{eq:pointwise}.
\end{proof}

\begin{cor}
The pointwise formula for the self-similar Laplacian on $SG_k$ for $x \in V^{\star} \setminus V_0$ is
$$\Delta_{\mu} u(x)=6 \lim\limits_{m \rightarrow \infty}{\left( \frac{H(q_1,q_2)}{2} \right)}^m \Delta_m u(x),$$
where $q_1,q_2$ are any two points on $V_0$.
\end{cor}

\begin{rem}
This gives us for the $SG_2$ that $\Delta_{\mu} u(x)=6 \lim\limits_{m \rightarrow \infty}{5}^m \Delta_m u(x)$ and for the $SG_3$ that $\Delta_{\mu} u(x)=6 \lim\limits_{m \rightarrow \infty}{\left(\frac{90}{7}\right)}^m \Delta_m u(x)$. 
\end{rem}

As for the energy Laplacian, a pointwise formula was computed in the unpublished thesis \cite{thesis} on $SG_2$. Here, we show that the formula is valid for all $SG_k$.

\begin{rem}
The reason for this is that while for the standard Laplacian the change is visible through the different rates of convergence, i.e., from $5^m$ to $\left(\frac{90}{7}\right)^m$ from $SG_2$ to $SG_3$, in the case of energy Laplacian the change occurs indirectly, owing to the different harmonic functions; for any $k$ in $SG_k$, we have different harmonic extension algorithms and thus the change occurs in the factor $\Delta_m ({h_1}^2+{h_2}^2)(x)$, since it is different for different values of $k$.
\end{rem}

\begin{prop}
Let $u \in dom\Delta_{\nu}$. Then for all $x \in V_* \diagdown V_0 $, we have 
$$\Delta_{\nu}u(x)=2\lim_{m\rightarrow \infty} \frac{\Delta_m u(x)}{\Delta_m ({h_1}^2 +{h_2}^2)(x)},$$
with uniform limit across $V_* \diagdown V_0$.
\end{prop}

\begin{proof}

It suffices to compute $\int_K {\psi_x^{(m)} \mathrm {d\nu}}$. We have from Theorem 2.4.2 of \cite{str} that for any measure $\mu$ and $u \in dom\Delta_{\mu}$ that
$$ \int_K {\psi_x^{(m)}(y)\Delta_{\mu}u(y)} \mathrm {d\mu}=r^{-m}deg(x) \Delta_m u(x).$$
In fact this is stated for $SG_2$ but the result is valid also for $SG_k$ because the proof remains identical due to the matching condition of the local normal derivatives which is that they sum to zero. Then by using the sum ${h_1}^2+{h_2}^2$ instead of $u$ and by using the energy Laplacian we get that
$$\int_K {\psi_x^{(m)}(y)\Delta_{\nu}({h_1}^2+{h_2}^2) \mathrm {d\nu(y)}}=r^{-m} deg(x)\Delta_m({h_1}^2+{h_2}^2).$$
Moreover, the non-degeneracy of the harmonic structure of $SG_k$ implies that all the conditions of Corollary 6.2 of \cite{harmcoord} are satisfied. This gives us, with a difference in normalization compared to \cite{harmcoord}, that $\Delta_v(h_1^2+h_2^2)=2$ from which the result follows.

\end{proof}

It was also shown in \cite{thesis} that in $SG_2$ the only functions that are in both the domains of the standard Laplacian and in the domain of the Kusuoka Laplacian are the harmonic functions. The proof can be generalized to all $SG_k$ with only a minor modification based on the slow decay rate of $r_k$. We present the entire proof here for the convenience of the reader.

\begin{prop}
If $u \in dom\Delta \cap dom\Delta_{\nu}$ on $SG_k$ then $u$ is harmonic.
\end{prop}
\begin{proof}
Let $u \in dom\Delta \cap dom\Delta_{\nu}$ such that $u$ is not harmonic and let $h_1,h_2$ be the orthonormal basis of harmonic functions modulo constants in the definition of the Kusuoka measure. Since the harmonic extension algorithm is non-degenerate \cite{tsougkas}, then as in \cite{thesis}, by the assumption of $u$ not being harmonic there exists a junction point $x$ such that $\Delta_{\nu} u(x) \neq 0$ and $\partial_{n}h_1(x) \neq 0$. The latter gives us that there exists a constant $c>0$ and $x_m$, a neighbor of $x$ in $\Gamma_m$, such that $r_k^{-m} |h_1(x_m)-h_1(x)| > c$. Let $\lambda= \frac{k(k+1)}{2r_k}$. We can check that $\Delta_m {h_2}^2(x)>0$ and also by the harmonicity of $h_1$ we can write
\begin{equation*}
\begin{split}
deg(x) \Delta_m {h_1}^2(x)&=\sum_{y\sim_m x} {(h_1(y)-h_1(x))}^2+2h_1(x)\sum_{y\sim_m x} (h_1(y)-h_1(x))\\
&=\sum_{y\sim_m x} {(h_1(y)-h_1(x))}^2.
\end{split}
\end{equation*}
Then we have that
\begin{equation*}
\begin{split}
 \lambda ^m \Delta_m ({h_1}^2 +{h_2}^2)(x) & \geq \lambda ^m \Delta_m {h_1(x)}^2=\frac{\lambda ^m}{deg(x)}\sum_{y\sim_m x} {(h_1(y)-h_1(x))}^2 \\
& \geq \frac{\lambda ^m}{deg(x)} {|h_1(x_m)-h_1(x)|}^2   \geq \frac{\lambda ^m}{deg(x)}r_k^{2m}c^2.
\end{split}
\end{equation*}
But by the discussion above, we see for $k>2$ that $\lambda r_k^2 >1$ and thus 
$$\lambda ^m \Delta_m {(h_1^2(x)+h_2^2(x)} \rightarrow \infty.$$
By using the pointwise formula of the energy Laplacian, i.e, $$\Delta_{\nu}u(x)=2\lim_{m\rightarrow \infty} \frac{\Delta_m u(x)}{\Delta_m ({h_1}^2 +{h_2}^2)(x)}$$
and the fact that $\lim_{m\rightarrow \infty} \lambda ^m \Delta_m u(x)< \infty$, we get that $\Delta_{\nu}u(x)=0$, a contradiction.
\end{proof}

\subsection{Vector self-similarity of the Kusuoka measure on $SG_k$}
Here we provide a self-similar identity for the energy Laplacian in $SG_k$ analogous to the self-similar identity derived for $SG_2$ in \cite{str3}. This is the type of self-similarity one can expect from the Kusuoka measure, since it is shown from the approach in \cite{JOP} that the Kusuoka measure generalizes Bernoulli measures on symbolic spaces to higher dimensions; we multiply matrices instead of numbers when we consider Kusuoka measures.

First of all the energy measures $\nu_{h_i}$ of the harmonic functions $h_i$ will be denoted for brevity as $\nu_i$. From some elementary computations it is found in \cite{str3} that
\begin{equation}
\label{eq:EnMeaRep}
\begin{split}
\nu_{h_0,h_1}&=\frac{1}{2}(-\nu_0-\nu_1+\nu_2);\\
\nu_{h_0,h_2}&=\frac{1}{2}(-\nu_0+\nu_1-\nu_2);\\
\nu_{h_1,h_2}&=\frac{1}{2}(\nu_0-\nu_1-\nu_2).
\end{split}
\end{equation}

Fix a natural number $k \geq 2$. To simplify notation we assume that all calculations in the remainder of this section are done in $SG_k$, which has a total of $d=\frac{k(k+1)}{2}$ cells and $r$ is the renormalization constant $r_k$. Denote by $p_{ni}^j$ the probability that a random walk on the $\Gamma_1$ graph starting at the vertex $F_n(q_i)$ first hits $V_0$ at $q_j$. 

Take the symmetric harmonic function $h_i$ with $1$ at the boundary vertex $q_i$ and zero on $V_0 \setminus \{q_i\}$. From the probabilistic interpretation of the harmonic extension algorithm on $SG_k$, we can establish the relations
\begin{equation}
\label{eq:harmext}
h_j \circ F_n= \sum_{i=0}^2 p_{ni}^j h_i.
\end{equation}

\begin{thm}
\label{thm:MMatrices}
In the case of $SG_k$, for every cell $C$ we have that
$$
\left(
\begin{array}{c}
\nu_0(F_i C)\\ \nu_1(F_i C) \\ \nu_2(F_i C)  
\end{array}
\right) =
M_i
\left(
\begin{array}{c}
\nu_0(C)\\ \nu_1(C) \\ \nu_2(C)
\end{array}
\right)
$$
and in particular, 
\begin{equation}
\label{eq:matrices}
\left(
\begin{array}{c}
\nu_0\\ \nu_1\\ \nu_2
\end{array}
\right) = \sum_{n=0}^{d-1} 
M_n
\left(
\begin{array}{c}
\nu_0\\ \nu_1 \\ \nu_2
\end{array}
\right)\circ F_n^{-1}.
\end{equation}
where $M_n=[\mu_{ji}^n]$ is the $3\times 3$ matrix with elements 
$$\mu_{j0}^n=\frac{1}{r}((p_{n0}^j)^2-p_{n0}^jp_{n1}^j-p_{n0}^jp_{n2}^j+p_{n1}^jp_{n2}^j),$$
$$\mu_{j1}^n=\frac{1}{r}((p_{n1}^j)^2-p_{n0}^jp_{n1}^j+p_{n0}^jp_{n2}^j-p_{n1}^jp_{n2}^j),$$
$$\mu_{j2}^n=\frac{1}{r}((p_{n2}^j)^2+p_{n0}^jp_{n1}^j-p_{n0}^jp_{n2}^j-p_{n1}^jp_{n2}^j).$$

\end{thm}

\begin{proof}
Let $n \in \{0, 1, \dots d-1\}$, $j\in \{0,1,2\}$ and $f$ be a continuous function on $SG_k$. Then by \eqref{eq:harmext} we see that
\begin{equation*}
\begin{split}
\int_{F_n K}f\, d\nu_j &=\frac{1}{r}\int_{K}f\circ F_n \,d\nu_{\sum_{i=0}^2 p_{ni}^j h_i}=\frac{1}{r}((p_{n0}^j)^2 \int_{K}f \circ F_n\, d\nu_0+(p_{n1}^j)^2 \int_{K}f \circ F_n\, d\nu_1 \\
&+(p_{n2}^j)^2 \int_{K}f \circ F_n\, d\nu_{2} + 2 p_{n0}^j p_{n1}^j\int_{K}f \circ F_n\, d\nu_{0,1}+ 2 p_{n0}^j p_{n2}^j \int_{K}f \circ F_n\, d\nu_{0,2} \\
& +  2 p_{n1}^j p_{n2}^j \int_{K}f \circ F_n\, d\nu_{1,2}). 
\end{split}
\end{equation*}
Then using \eqref{eq:EnMeaRep} we arrive at $\int_{F_nK} f\, d\nu_j=\sum_{n} {\mu_{ni}^j \int_K f\circ F_n \, d\nu_i}$
from which we deduce \eqref{eq:matrices}
\end{proof}

Let $S_i^n$ be the sum of the elements of the column $i$ in $M_n$ and $R_i=\frac{d\nu_i}{d\nu}$. Using the notation $Q_j =  \sum_{i=0}^2 S_i^j R_i$ we obtain the following consequence.
\begin{cor}
 The Kusuoka measure satisfies the variable weight self-similar identity
\begin{equation}
\label{eq:nuSelfSimilar}
\nu=  \sum_{n=0}^{d-1} Q_n \nu \circ F_n^{-1}
\end{equation}
and thus for any integrable $f$ we have that 
\begin{equation}
\label{eq:nuSelfSimilar2}
\int_K f \, d\nu=  \sum_{n=0}^{d-1}\int_K Q_n f \circ F_n\, d\nu .
\end{equation} 
\end{cor}
\begin{proof}
The result follows immediately by using equation (2.5) and by the fact that $\nu=\left( \begin{array}{ccc} 1&1&1 \end{array}\right)\left( \begin{array}{c} \nu_0\\ \nu_1\\ \nu_2 \end{array}\right)$.
\end{proof}
From these observations we now obtain the ``self-similar" scaling formula for the energy Laplacian $\Delta_\nu$ on the $SG_k$.

\begin{thm}
The energy Laplacian $\Delta_\nu$ satisfies
\begin{equation}
\label{eq:lapselfsimilar}
\Delta_\nu (u \circ F_j) = r Q_j (\Delta_\nu u) \circ F_j
\end{equation}
$\nu$-almost everywhere.
\end{thm}

\begin{proof}
The proof follows the exact same methodology of \cite{str3}. By the definition of the energy Laplacian and by \eqref{eq:nuSelfSimilar2} with $f$ now being $(\Delta_{\nu}u)v$ we have,
\begin{align*}
-\mathcal{E}(u,v) &= \int_{K}{(\Delta_{\nu}u)v d\nu}\\
&=\sum_{j=0}^{d-1}{\int_{K}{\sum_{i=0}^2 S_i^n \frac{d\nu_i}{d\nu}(\Delta_{\nu}u)\circ F_j\hspace{6pt} v \circ F_j}\hspace{6pt}d\nu}. \\
\end{align*}
We also know that $\mathcal{E}$ is self-similar, hence
$$\mathcal{E}(u,v) = \frac{1}{r} \sum_{j=0}^{d-1}{\mathcal{E}(u\circ F_j, v\circ F_j)}.$$
Combining these and using the fact that $v$ is arbitrary establishes \eqref{eq:lapselfsimilar}
\end{proof}
By iterating \eqref{eq:lapselfsimilar} we have the following result which was proven in \cite{str3} for the case $k=2$.

\begin{cor}
Let $w = (w_1,...,w_m)$ be a finite word of length $m$ and $F_w = F_{w_1}\circ....\circ F_{w_m}$. Define
$$
Q_w = Q_{w_m} \cdot (Q_{w_{m-1}}\circ F_{w_m}) \cdot (Q_{w_{m-2}}\circ F_{w_{m-1}}\circ F_{w_m})\cdots (Q_{w_1}\circ F_{w_2}\circ \cdots \circ F_{w_m})
$$
Then
$$
\Delta_\nu (u\circ F_w) =r^m Q_w (\Delta_\nu u) \circ F_w
$$
$\nu$-almost everywhere.
\end{cor}
As an example, we compute the scaling formula for $SG_3$. The probabilities for the random walk can be easily evaluated by the ``$\frac{1}{3}-\frac{4}{15}-\frac{8}{15}$" rule and the matrices are 
$$M_0 = \frac{1}{105}\left(
\begin{array}{ccc}
49&0&0\\
12&4&-3\\
12&-3&4
\end{array}
\right),
M_1 = \frac{1}{105}\left(
\begin{array}{ccc}
4&12&-3\\
0&49&0\\
-3&12&4
\end{array}
\right),$$

$$M_2 = \frac{1}{105}\left(
\begin{array}{ccc}
4&-3&12\\
-3&4&12\\
0&0&49
\end{array}
\right),
M_3 = \frac{1}{105}\left(
\begin{array}{ccc}
4&0&0\\
-3&12&4\\
-3&4&12
\end{array}
\right),$$

$$M_4 = \frac{1}{105}\left(
\begin{array}{ccc}
12&-3&4\\
0&4&0\\
4&-3&12
\end{array}
\right),
M_5 = \frac{1}{105}\left(
\begin{array}{ccc}
12&4&-3\\
4&12&-3\\
0&0&4
\end{array}
\right),
$$
which gives us
$$\int_K f \, d\nu= \frac{1}{105} \sum_{i=0}^2\int_K \left(1+72R_i \right) f \circ F_i\, d\nu + \frac{1}{105}\sum_{i=3}^5 \int_K \left(16-18R_{i-3} \right) f\circ F_i \, d\nu$$
and also
\begin{equation*}
\begin{split}
&Q_j= \frac{1}{105}\left( 1 + 72R_j\right) \text{  for  }  j=0,1,2,\\
&Q_j = \frac{1}{105}(16-18R_{j-3}) \text{  for  } j=3,4,5.
\end{split}
\end{equation*}

For general $k$ we can use the symmetry of $SG_k$ to simplify some of these formulas. It is common practice to write the $F_i$'s in groups of three such that the cells $F_iK$ are counterclockwise rotations of each other by $2\pi /3$. Then the equivalent matrices are related in the sense that it suffices to compute only one of these. For the other ones in that group of three we have that $\mu_{ij}^{k+1}=\mu_{i-1 j-1}^k$ where the operations are in $\mathbb{Z}_3$. A nice simplification happens in $SG_k$'s for $k$ such that there exists a cell exactly in the middle, let's call it $F_cK$, which is invariant under rotations. Those $k$'s are of the form $k=3l+1$ for $l \in \mathbb{N}^{*}$. Then all the $S_i$'s are equal to each other and hence there is no need for the use of Radon--Nikodym derivatives. As a consequence, the following scaling formula will be valid everywhere for the Laplacian on that cell
$$\Delta_\nu (u\circ F_c) = \frac{1}{2} \mathcal{E} (h_1 \circ F_c) (\Delta_\nu u) \circ F_c.$$

\end{document}